\documentclass[twoside, 12pt]{article}
\usepackage{mathrsfs}
\usepackage{amssymb, amsmath, mathrsfs, amsthm}
\usepackage{graphicx}
\usepackage{amsmath}
\usepackage{color}
\usepackage[left]{lineno}
\usepackage[top=2cm, bottom=2cm, left=2cm, right=2cm]{geometry}
\usepackage{float, caption}

\input{epsf}

\newcommand{\bd}{\begin{description}}
\newcommand{\ed}{\end{description}}
\newcommand{\bi}{\begin{itemize}}
\newcommand{\ei}{\end{itemize}}
\newcommand{\be}{\begin{enumerate}}
\newcommand{\ee}{\end{enumerate}}
\newcommand{\beq}{\begin{equation}}
\newcommand{\eeq}{\end{equation}}
\newcommand{\beqs}{\begin{eqnarray*}}
\newcommand{\eeqs}{\end{eqnarray*}}

\numberwithin{equation}{section}

\definecolor{DarkGreen}{rgb}{0.2, 0.6, 0.3}

\newcommand{\Tt}{{\mathcal{T}}}
\newcommand{\Tc}{{\overline{\mathcal{T}}}}

\newtheorem{theorem}{Theorem}[section]

\newtheorem{lemma}[theorem]{Lemma}
\newtheorem{definition}[theorem]{Definition}
\newtheorem{corollary}[theorem]{Corollary}
\newtheorem{case}{Case}

\newtheorem{remark}[theorem]{Remark}

\newtheorem{fact}{Fact}
\newtheorem{proposition}[theorem]{Proposition}

\begin{document}
\title{\textbf{Ramsey numbers for partially-ordered sets
} \footnote{
The work was supported by the Hungarian National
Research, Development and Innovation Office
NKFIH (No.~SSN135643 and K132696); the National Science Foundation of China
(Nos. 12471329,~12301458~and~12061059); and the Zhejiang Provincial Natural Science Foundation of China (No. LQ24A010005); JSPS KAKENHI (Grant Numbers JP22K19773 and JP23K03195)}}

\author{Gyula O.H. Katona\thanks{HUN-REN Alfr\'ed R\'enyi Institute of Mathematics, Budapest Reaaltanoda utca 13-15, 1053, Hungary. {\tt
katona.gyula.oh@renyi.hu}}, \ \ Yaping Mao\thanks{Corresponding author: Academy of Plateau Science and Sustainability,
and School of Mathematics and Statistics, Qinghai Normal University, Xining, Qinghai 810008, China. {\tt yapingmao@outlook.com; myp@qhnu.edu.cn}} , \ \  Kenta Ozeki\thanks{Faculty of Environment and Information Sciences, Yokohama
National University, 79-2 Tokiwadai, Hodogaya-ku, Yokohama 240-8501,
Japan. {\tt ozeki-kenta-xr@ynu.ac.jp}},
\ \ Zhao
Wang\thanks	{College of Science, China Jiliang University, Hangzhou
310018, China. {\tt
wangzhao@mail.bnu.edu.cn}}}
\date{}
\maketitle

\begin{abstract}
We say that a poset $Q$
contains a copy (resp.~an induced copy) of a poset $P$
if there is an injection $f : P \to Q$
such that for any $x,y \in P$,
$f(x)\leq f(y)$ in $Q$ if (resp.~if and only if) $x\leq y$ in $P$.
Let $\mathcal{Q}=\{Q_{n} : n\geq 1\}$ be a
family of posets such that $Q_n\subseteq Q_{n+1}$ and $|Q_n|<|Q_{n+1}|$ for each $n$.
For given $k$ posets $P_1, P_2, \dots , P_k$,
the \emph{weak (resp.~strong) poset Ramsey number for $t$-chains}
is the smallest number $n$ such that for any coloring of $t$-chains in $Q_n\in \mathcal{Q}$ with $k$ colors, say $1,2, \dots, k$,
there is a monochromatic (resp.~induced) copy of the poset $P_i$ in color $i$ for some $1\leq i\leq k$.
In this paper,
we give several lower and upper bounds on 
the weak and strong poset Ramsey number for $t$-chains.
%
\\[2mm]
\\[2mm]
\end{abstract}

\section{Introduction}

Ramsey theory
is a branch of mathematics that studies the conditions of when a combinatorial object necessarily contains some smaller given objects; see \cite{Ramsey30}. This subject has been
a hot topic in mathematics for nearly one century due to their intrinsic
beauty, wide applicability, and overwhelming difficulty despite
somewhat misleadingly simple statements.
Ramsey theory has important applications in algebra, geometry, logic, ergodic theory, theoretical computer science, information theory, number theory, and set theory; see the book \cite{GRS90} and survey paper \cite{Rosta04}.

A \emph{partially-ordered set}, or a \emph{poset}, is a pair $(P,\leq )$,
where $P$ is a set and $\leq $ is a relation on $P$
that is reflexive, anti-symmetric, and transitive.
When the relation $\leq$ is clear from the context, we simply write $P$ as a poset.
A pair $x, y\in P$ is \emph{comparable} if $x\leq y$ or $y\leq x$.
A subset $\mathcal{C}\subseteq P$ is called a \emph{chain} if any two elements
of $\mathcal{C}$ are comparable, and a subset $\mathcal{A}\subseteq P$ is called
an \emph{anti-chain} if any two distinct elements of $\mathcal{A}$ are incomparable.
A \emph{$t$-chain} is a set of $t$ distinct pairwisely comparable elements.
Let $|P|$ denote the number of elements in $P$.

If $P$ and $Q$ are posets, then an injection $f : P\to Q$ is a \emph{weak embedding} if $f(x)\leq f(y)$ when
$x\leq y$; we say that $f(P)$ is a copy of $P$ in $Q$.
An injection $f : P\to Q$ is a \emph{strong embedding} if $f (x)\leq f (y)$ if and
only if $x\leq y$; we say that $f(P)$ is an \emph{induced copy} of $P$ in $Q$.
In this paper, we consider a $k$-coloring of $t$-chains in a given poset,
which is an assignment of $k$ colors,
say $1,2, \dots , k$,
to each $t$-chain in the poset.
If $t=1$, that is,
if it is a $k$-coloring of elements of the poset,
then it is simply called a $k$-coloring of the poset.
A colored poset is \textit{monochromatic} if all of its $t$-chains
share the same color.

Ramsey theory on posets was initiated by Ne\v{s}et\v{r}il and
R\"{o}dl \cite{NesetrilRodl}, who studied colorings of induced copies
of a fixed poset.
Kierstead and Trotter \cite{KiersteadTrotter} subsequently considered
arbitrary host posets instead of Boolean lattices and investigated
Ramsey-type problems in terms of the cardinality, height and width
of the host.
For further developments in this direction we refer to
\cite{DKT91, KiersteadTrotter, McColm, TrotterRamsey}.
More recently, several authors
\cite{AW17, CS18, GundersonRodlSidorenko, JohnstonLuMilans, LT22}
have focused specifically on finding induced copies of posets
in the Boolean lattice.

Cox and Stolee \cite{CS18} recast Ramsey theory for partially ordered
sets in the language of \emph{pographs}, that is, partially ordered
hypergraphs.
Let $(Q,\le)$ be a poset and let $t\ge1$.
A \emph{$t$-uniform pograph on $Q$} is a $t$-uniform hypergraph $H$
with vertex set $Q$ such that every edge of $H$ is a $t$-chain in $Q$
(we do not require that all $t$-chains of $Q$ appear as edges).
In this framework the objects of interest are Ramsey numbers for such
pographs.
For the purposes of this paper we use the following equivalent
formulation in terms of posets.

\begin{definition}\label{Defi-Wposet}
For a family $\mathcal{Q}=\{Q_{n} : n\geq 1\}$ of host posets
such that $Q_n \subseteq Q_{n+1}$ and $|Q_n|<|Q_{n+1}|$ for each $n$,
and given $k$ posets $P_1,P_2,\ldots,P_k$,
the \emph{weak poset Ramsey number for $t$-chains},
denoted by $\operatorname{R}_{k,t}(\mathcal{Q}\,|\,P_1,P_2,\ldots,P_k)$,
is the smallest integer $N$
such that for every $k$-coloring of the $t$-chains in $Q_N\in\mathcal{Q}$,
there exists a monochromatic copy of $P_i$ in color $i$ for some
$1\leq i\leq k$.
If no such $N$ exists, we set
\[
\operatorname{R}_{k,t}(\mathcal{Q}\,|\,P_1,\ldots,P_k)=\infty.
\]
\end{definition}

We simply write $\operatorname{R}_{k,t}(\mathcal{Q}\,|\,P)$ for $\operatorname{R}_{k,t}(\mathcal{Q}\,|\,P_1,P_2,\ldots,P_k)$ if $P_1=\cdots=P_k=P$.
If $t=1$, then we write $\operatorname{R}_{k}(\mathcal{Q}\,|\,P_1,P_2,\ldots,P_k)$. Furthermore, if $t=1$ and $k=2$, then we write $\operatorname{R}(\mathcal{Q}\,|\,P_1,P_2)$.

We analogously define the \emph{strong poset Ramsey number} $\operatorname{R}_{k,t}^{\sharp}(\mathcal{Q}\,|\,P_1,P_2,\ldots,P_k)$ for induced copy of $P_i$.
In addition,
we use the similar abbreviations such as
$\operatorname{R}_{k,t}^{\sharp}(\mathcal{Q}\,|\,P)$,
$\operatorname{R}_{k}^{\sharp}(\mathcal{Q}\,|\,P_1, P_2, \dots , P_k)$,
and
$\operatorname{R}^{\sharp}(\mathcal{Q}\,|\,P_1, P_2)$.

As we can see below,
these concepts include some known poset Ramsey problems.

\begin{itemize}
\item
A \emph{Boolean lattice of dimension $n$}, denoted $B_n$, is the
power set of an $n$-element ground set equipped with inclusion relation.
Note that $|B_n| = 2^n$.
In this paper,
unless otherwise noted,
we use the set $[n]=\{1,2,...,n\}$
as a ground set of the Boolean lattice $B_n$ of dimension $n$.
Let $\mathcal{B}=\{B_{n} : n\geq 1\}$ be
the family of Boolean lattices.
Then
$\operatorname{R}_{k,t}(\mathcal{B}\,|\,P_1,P_2,\ldots,P_k)$
is called the \emph{Boolean Ramsey number}, see \cite{CS18}.

\item
Similarly,
the variation of strong embeddings has been considered:
in \cite{AW17, LT22},
$\operatorname{R}_{k,t}^{\sharp}(\mathcal{B}\,|\,P_1,P_2,\ldots,P_k)$
is called simply
the \emph{poset Ramsey number}.

\item
Let $\mathcal{C}=\{C_{n} : n\geq 1\}$ be a
family of the chains,
where $C_n$ denotes the $n$-chain,
see Figure \ref{some_posets_fig}.
Then
$\operatorname{R}_{k,t}(\mathcal{C}\,|\,P_1,P_2,\ldots,P_k)$
is called \emph{the chain Ramsey number},
see
\cite{CS18}.
If $t=1$, then the chain Ramsey number can be trivially found by the Pigeonhole principle,
and hence we assume $t\geq 2$ for considering this parameter.

Chain Ramsey number is an extension of the ordered Ramsey number for hypergraphs.
In particular,
when $t=2$,
the chain Ramsey number corresponds to
the ordered Ramsey number for graphs.
For more details on the ordered Ramsey numbers,
we refer to \cite{BJV19, CFLS17}.

\end{itemize}

We will focus mainly on the following posets,
see Figure \ref{some_posets_fig}.
\begin{itemize}
\item The \emph{$n$-anti-chain}, denoted $A_n$,
is the poset in which any two elements are not comparable.

\item The \emph{$s$-matching}, denoted $M_s$,
is the poset with elements $\{x_1,\ldots,x_s\}\cup \{y_1,\ldots,y_s\}$
where $x_i\leq y_i$ for all $i$,
and all other pairs are not comparable.

\item The \emph{$(r,s)$-butterfly}, denoted by $\bowtie^s_r$,
is
the poset with elements $\{x_1,\ldots,x_r\}\cup \{y_1,\ldots,y_s\}$
where $x_i\leq y_j$ for all $i$ and $j$
and all other pairs are not comparable.
We use $\bowtie$ to denote $\bowtie^2_2$.

\item 
The \emph{$r$-diamond}, denoted by $\lozenge_r$, is
the poset with elements $\{x,y_1,\ldots,y_r,z\}$ such that
$x\leq y_i\leq z$ for all $i$ and, for all $1\leq i<j\leq r$,
the elements $y_i$ and $y_j$ are incomparable.

\end{itemize}

\begin{figure}
\begin{center}
\begin{tabular}{|c|c|c|c|c|}
\hline $
\begin{array}{l}
\setlength{\unitlength}{1mm}
\begin{picture}(20,25)
\unitlength 1mm \linethickness{0.4pt}

\put(10,23){\circle*{1.20}}\put(10,18){\circle*{1.20}}\put(10,13){\circle*{1.20}}
\put(10,7){\circle*{1.20}}\put(10,2){\circle*{1.20}}

\put(10,11){\circle*{0.50}}\put(10,10){\circle*{0.50}}
\put(10,9){\circle*{0.50}}

\put(10,23){\line(0,-1){10}}\put(10,7){\line(0,-1){5}}

\put(14,2){\makebox(0,0)[cc]{$x_1$}}\put(14,7){\makebox(0,0)[cc]{$x_{2}$}}\put(14,23){\makebox(0,0)[cc]{$x_n$}}
\put(16,18){\makebox(0,0)[cc]{$x_{n-1}$}}
\end{picture}
\end{array}
$& $
\begin{array}{l}
\setlength{\unitlength}{1mm}
\begin{picture}(20,25)
\unitlength 1mm \linethickness{0.4pt}

\put(0,15){\circle*{1.20}}\put(5,15){\circle*{1.20}}
\put(10,15){\circle*{1.20}}\put(18,15){\circle*{1.20}}

\put(13,15){\circle*{0.50}}\put(14,15){\circle*{0.50}}\put(15,15){\circle*{0.50}}


\put(1,12){\makebox(0,0)[cc]{$x_1$}}\put(6,12){\makebox(0,0)[cc]{$x_2$}}\put(11,12){\makebox(0,0)[cc]{$x_3$}}
\put(19,12){\makebox(0,0)[cc]{$x_n$}}
\end{picture}
\end{array}
$& $
\begin{array}{l}
\setlength{\unitlength}{1mm}
\begin{picture}(20,25)
\unitlength 1mm \linethickness{0.4pt}

\put(1,18){\circle*{1.00}}\put(8,18){\circle*{1.00}}\put(8,8){\circle*{1.00}}
\put(18,18){\circle*{1.00}}\put(18,8){\circle*{1.00}}\put(1,8){\circle*{1.00}}

\put(8,18){\line(0,-1){10}}\put(18,18){\line(0,-1){10}}\put(1,18){\line(0,-1){10}}

\put(14,13){\circle*{0.50}}\put(13,13){\circle*{0.50}}\put(12,13){\circle*{0.50}}

\put(2,21){\makebox(0,0)[cc]{$y_1$}}\put(9,21){\makebox(0,0)[cc]{$y_2$}}\put(18,21){\makebox(0,0)[cc]{$y_s$}}
\put(2,5){\makebox(0,0)[cc]{$x_1$}}\put(9,5){\makebox(0,0)[cc]{$x_2$}}\put(18,5){\makebox(0,0)[cc]{$x_s$}}

\end{picture}
\end{array}
$ & $
\begin{array}{l}
\setlength{\unitlength}{1mm}
\begin{picture}(20,25)
\unitlength 1mm \linethickness{0.4pt}

\put(1,18){\circle*{1.00}}\put(8,18){\circle*{1.00}}\put(8,8){\circle*{1.00}}
\put(20,18){\circle*{1.00}}\put(20,8){\circle*{1.00}}\put(1,8){\circle*{1.00}}

\put(8,18){\line(0,-1){10}}\put(1,8){\line(2,1){19}}
\put(20,18){\line(0,-1){10}}\put(1,18){\line(0,-1){10}}
\put(1,18){\line(2,-1){19}}\put(1,8){\line(2,3){7}}\put(8,8){\line(-2,3){7}}
\put(8,8){\line(5,4){12}}\put(8,18){\line(5,-4){12}}

\put(14,18){\circle*{0.50}}\put(13,18){\circle*{0.50}}\put(12,18){\circle*{0.50}}
\put(14,8){\circle*{0.50}}\put(13,8){\circle*{0.50}}\put(12,8){\circle*{0.50}}

\put(2,21){\makebox(0,0)[cc]{$y_1$}}\put(9,21){\makebox(0,0)[cc]{$y_2$}}\put(20,21){\makebox(0,0)[cc]{$y_s$}}
\put(2,5){\makebox(0,0)[cc]{$x_1$}}\put(9,5){\makebox(0,0)[cc]{$x_2$}}\put(20,5){\makebox(0,0)[cc]{$x_r$}}

\end{picture}
\end{array}
$ & $\begin{array}{l} \setlength{\unitlength}{1mm}
\begin{picture}(20,25)
\unitlength 1mm \linethickness{0.4pt}

\put(13,17){\circle*{1.20}}\put(8,12){\circle*{1.20}}\put(4,12){\circle*{1.20}}
\put(18,12){\circle*{1.20}}\put(13,7){\circle*{1.20}}\put(15,12){\circle*{0.50}}
\put(16,12){\circle*{0.50}}\put(14,12){\circle*{0.50}}

\put(13,17){\line(-1,-1){5}}\put(13,17){\line(1,-1){5}}
\put(13,7){\line(-1,1){5}}\put(13,7){\line(1,1){5}}
\put(4,12){\line(5,3){9}}\put(4,12){\line(5,-3){9}}

\put(12,19){\makebox(0,0)[cc]{$z$}}\put(1,12){\makebox(0,0)[cc]{$y_1$}}
\put(12,12){\makebox(0,0)[cc]{$y_2$}}\put(21,12){\makebox(0,0)[cc]{$y_r$}}
\put(12,5){\makebox(0,0)[cc]{$x$}}
\end{picture}
\end{array}
$\\
\cline{1-5} $n$-chain & $n$-anti-chain  & $s$-matching &
$(r,s)$-butterfly & $r$-diamond \\
\cline{1-5} $C_n$ & $A_n$  & $M_s$ &
$\bowtie^s_r$ & $\lozenge_r$ \\
\cline{1-5}
\end{tabular}
\end{center}
\caption{Some posets.}
\label{some_posets_fig}
\end{figure}

We now introduce some known results for poset Ramsey numbers.
For the case of a $2$-coloring (i.e.~$k=2$) of a set of a Boolean poset (i.e.~$t=1$),
the Boolean Ramsey number $\operatorname{R}^{\sharp}(\mathcal{B}\,|\,B_{n},B_{m})$
for Boolean lattices $B_n$ and $B_m$
has been studied by
Axenovich and Walzer \cite{AW17, AW23}, Lu and Thompson \cite{LT22},
Gr\'{o}sz, Methuku, and Tompkins \cite{GMT23}, Bohman and Peng \cite{BP2021}, and Walzer \cite{Walzer15}.
The currently best-known bounds for $\operatorname{R}^{\sharp}(\mathcal{B}\,|\,B_{m},B_{n})$
are 
$\operatorname{R}^{\sharp}(\mathcal{B}\,|\,B_{n},B_{n})
\leq n^2 -n +2$
for $n \geq 3$
and
$\operatorname{R}^{\sharp}(\mathcal{B}\,|\,B_{m},B_{n})
\leq \lceil (m-1 + \frac{2}{m+1})n + \frac{1}{3}m + 2\rceil$
for $n > m \geq 4$,
obtained by
Lu and Thompson \cite{LT22}.
%
%
%
Axenovich and Winter \cite{AW23, AW24} and Winter \cite{Winter23, WinterII, WinterIII}
studied the behavior of the Boolean Ramsey number $\operatorname{R}^{\sharp}(\mathcal{B}\,| P,B_n)$ for a fixed
poset $P$ and a Boolean lattice $B_n$, when $n$ grows sufficiently large.
Falgas-Ravry et al.~\cite{FMTZ20}
considered Ramsey properties of random posets.

The case of $k$-coloring
with $k \geq 3$ is less known compared with the $2$-colored case.
The following are some known results for the case $t=1$.
For a Boolean lattice $B_n$ of dimension $n$,
the \emph{$i$-th level} of $B_n$ is the collection of all $i$-element subsets of $[n]$, denoted by ${[n]\choose i}$.
For a poset $P$, we write $e(P)$ for the largest $m$
such that $P$ cannot be embedded into any $m$ consecutive levels of $B_{n}$
for any $n$.

\begin{theorem}[{\upshape Cox and Stolee \cite{CS18}}]
\label{th-2-1}
Let $k$ be an integer with $k\geq 2$.
\begin{itemize}
\item[] $(i)$ $\operatorname{R}_{k}(\mathcal{B}\,|\,\bowtie)=2k+1$.

\item[] $(ii)$ $\max\{M,\sum_{i=1}^ke(P_i)\}\leq \operatorname{R}_{k}(\mathcal{B}\,|\,P_1,P_2,\ldots,P_k)\leq \sum_{i=1}^k(|P_i|-1)$,
where $M$ is the least integer with $P_i \subseteq B_M$ for all $i$.
\end{itemize}
\end{theorem}

Note that the
$2$-diamond $\lozenge_2$ is a Boolean lattice $B_2$.

\begin{theorem}[{\upshape Walzer \cite{Walzer15}}]\label{th-Walzer}
Let $k$ be an integer with $k\geq 2$.
\begin{itemize}
\item[] $(i)$ $\operatorname{R}_2^{\sharp}(\mathcal{B}\,|\,\lozenge_2)=4$.

\item[] $(ii)$ $2k \leq \operatorname{R}_k^{\sharp}(\mathcal{B}\,|\,B_2)=\operatorname{R}_2^{\sharp}(\mathcal{B}\,|\,\lozenge_2)\leq 3^k-1$.

\item[] $(iii)$ For $m\geq \max\{3,k\}$, we have
$mk\leq \operatorname{R}_{k}^{\sharp}(\mathcal{B}\,|\,B_m)\leq
\frac{(m-1)^{k+2}-(m-1)^{2}-k(m-2)}{(m-2)^{2}}$.
\end{itemize}
\end{theorem}

Cox and Stolee \cite[Section 2]{CS18}
in addition gave upper bounds of the Boolean Ramsey numbers
for some particular posets with $k$ colors.

Cox and Stolee \cite[Section 3]{CS18} also 
studied the case of $k$-coloring of $t$-chains in a Boolean lattice
with $k \geq 2$ and $t \geq 2$,
mainly for particular posets $P$ with $t = 2$.
However, as long as the authors know,
there are no other work on the case with $k \geq 3$ and $t \geq 2$,
except for the chain Ramsey number
from the context of
ordered Ramsey number of hypergraphs.
With this situation in mind,
in this paper, we study poset Ramsey numbers for some cases.
The following are our main contributions,
where we consider general $k$ and $t$ for 1.~and 3.~,
and general $k$ with $t=1$ for the others.

\begin{itemize}


\item[1.]
{\bf [Section \ref{lowerprob_sec}] Lower bounds for weak poset Ramsey numbers:}
\\
For a family $\mathcal{Q}=\{Q_{n} : n\geq 1\}$ of host posets
such that $Q_n \subseteq Q_{n+1}$
and $|Q_n|<|Q_{n+1}|$
for each $n$,
$k\geq 2$ and posets $P_1,P_2,...,P_k$,
we give a lower bound of $\operatorname{R}_{k,t}(\mathcal{Q}\,|\,P_1,P_2,...,P_k)$
in terms of $|P_k|$ and the number of its $t$-chains,
where $|P_k| = \max \{|P_i| \,|\, 1 \leq i \leq k\}$
(Theorem~\ref{th-Lower-R-general}) by Lov\'{a}sz Local Lemma.
There are only a few known results
for weak poset Ramsey number for $t$-chains,
and
this theorem gives the first general lower bound.

\item[2.]
{\bf [Section \ref{M2_sec}] Upper bounds for weak poset Ramsey numbers:}
\\
In Section \ref{M2_sec},
we give an upper bound of $\operatorname{R}_k(\mathcal{B}\,|\,P)$
by the Lubell function (Theorem \ref{th-condition}),
and prove
$\operatorname{R}_k(\mathcal{B}\,|\,M_2)=k+2$
and
$k+2\leq \operatorname{R}_k(\mathcal{B}\,|\,M_s)\leq \max\{k+7,s\}$ for $s \geq 3$ (Theorem \ref{th-Ms}). This gives an answer to the question
by Cox and Stolee \cite[Question 6.2]{CS18}
for $P = M_2$,
where
they asked,
for a poset $P$, whether there is a constant $c=c(P)$ such that
$\operatorname{R}_k(\mathcal{B}\,|\,P) = k\cdot e(P)+c$.

\item[3.]
{\bf [Section \ref{lowerB_sec}] Lower bounds for strong poset Ramsey numbers for $B_m$:}
\\
Adapting the method for graph Ramsey numbers due to Conlon-Ferber \cite{CF21}, Wigderson \cite{Wi21} and Sawin \cite{Sa22},
we give a general lower bound for Boolean Ramsey number for $k \geq 3$ and $t \geq 2$ (Theorem \ref{th-lower-Boolean}),
which implies $\operatorname{R}_{k,t}^{\sharp}(\mathcal{B}\,|\,B_{m_1},B_{m_2},\ldots,B_{m_k}) = \Omega\left((\log_2 k) t^{m_1}+  t^{m_k}\right)$,
where $t-1 \leq m_1 \leq \cdots  \leq m_k$.

\item[4.]
{\bf [Section \ref{boo_subsec}] Upper bounds for strong poset Ramsey numbers for $B_m$}
\\
In Section \ref{boo_subsec}, 
we show 
$$
\operatorname{R}_k^{\sharp}(\mathcal{B}\,|\,B_{m})\leq (\operatorname{R}_{\lfloor k/2\rfloor}^{\sharp}(\mathcal{B}\,|\,B_{m})-2)\operatorname{R}_{\lceil k/2\rceil}^{\sharp}(\mathcal{B}\,|\,B_{m})+\operatorname{R}_{\lfloor k/2\rfloor}^{\sharp}(\mathcal{B}\,|\,B_{m}).
$$
for $k\geq 6$ and $m\geq 2$.
It is worth pointing out that 
Walzer \cite[page 68]{Walzer15} obtained the following recurrence relation:
\begin{equation*}
\begin{split}
{\rm R}_k^{\sharp}(\mathcal{B}\,|\,B_{m})
&\leq (m-1){\rm R}_{k-1}^{\sharp}(\mathcal{B}\,|\,B_{m})+m+k-1
\end{split}
\end{equation*}
for $k\geq 3$.
Considering the
bound $\operatorname{R}^{\sharp}_2(\mathcal{B}\,|\,B_{m})\leq m^2-(1-\epsilon)m\log m$
by \cite[Corollary 6.2]{Winterphd}
for $\epsilon>0$ and sufficiently large $m$,
our recurrence relation gives an upper bound
for $\operatorname{R}^{\sharp}_k(\mathcal{B}\,|\,B_{m})$,
which is better than Theorem \ref{th-Walzer} (iii),
with the same leading-order behavior as those of Walzer,
but improves the second-largest term in the expression.


\item[5.]
{\bf [Section \ref{dia_subsec}] Upper bounds for strong poset Ramsey numbers for $r$-diamonds:}
\\
As in Theorem \ref{th-Walzer} (ii),
Walzer \cite{Walzer15} proved that $2k
\leq \operatorname{R}_k^{\sharp}(\mathcal{B}\,|\,B_2)=\operatorname{R}_k^{\sharp}(\mathcal{B}\,|\,\lozenge_2)\leq 3^k-1$.
In Section \ref{dia_subsec},
we give an upper bound of  $\operatorname{R}_k^{\sharp}(\mathcal{B}\,|\,\lozenge_r)$
for general $r \ge 2$
(Theorem \ref{BR22-B2}),
which, in the case $r=2$, improves Walzer's exponential upper bound
$3^k-1$ to the linear bound $5k-3$,
providing strong evidence that
$\operatorname{R}_k^{\sharp}(\mathcal{B}\,|\,\lozenge_2)$
grows linearly in $k$.

\end{itemize}

\section{Lower bounds for weak poset Ramsey numbers}
\label{lowerprob_sec}

\subsection{Lower bounds obtained from Lov\'{a}sz Local Lemma}

The probabilistic method is a powerful technique for approaching
asymptotic combinatorial problems.
In particular, Lov\'{a}sz Local Lemma is a good tool to study Ramsey problems, which was used to give lower bounds of graph Ramsey number of complete graphs; see the paper \cite{Sp77} and the book \cite[P.72]{APE92}.
In \cite{CS18}, Cox and Stolee gave a lower bound for Boolean Ramsey numbers,
where the host posets are also Boolean lattices.
One of the aims for this paper is to give a lower bound to the poset Ramsey numbers
by Lov\'{a}sz Local Lemma,
where both of the host posets in $\mathcal{Q}$ and given posets $P_1,P_2,\ldots,P_k$ in Definition \ref{Defi-Wposet} are general and the objects are $k$-colorings of $t$-chains
with $k \geq 2$ and $t \geq 2$.

\begin{theorem}\label{th-Lower-R-general}
For $k\geq 2$, 
let $P_1,P_2,\dots,P_k$ be posets in which every element lies in at least one
$t$-chain,
and let $n_i = |P_i|$ and $m_i$ be the number of $t$-chains in $P_i$ 
for $1\leq i\leq k$, respectively, where $2\leq t\leq n_1\leq n_2\leq ...\leq n_k$, and
$3 \leq n_k$.
Let $m=\min\{m_1,m_2,\dots,m_{k-1}\}$ with $m\geq 1$ and $d = {n_k\choose t}-m_k$.
Then, there exists a constant $C_0$
depending on
$k,t,P_1,\dots ,P_{k-1}$
such that
if $n$ satisfies
\begin{eqnarray*}
\ln n < \frac{m_k}{n_k+td+2}
\quad \text{and} \quad (\ln n)^{m} n^{n_{k-1}} < C_0 \left(\frac{m_k}{n_k+td+2}\right)^m,
\end{eqnarray*}
then
for any poset $Q$ with $n$ elements,
there exists a $k$-coloring of the $t$-chains in $Q$
with no monochromatic copy of $P_i$ in color $i$ for any $1\leq i\leq k$.
\end{theorem}

Note that $C_0$ does not depend on $P_k$.
Therefore,
Theorem \ref{th-Lower-R-general} implies the following.
Let $\mathcal{Q}=\{Q_{s} : s \geq 1\}$ such that $Q_{s} \subseteq Q_{s+1}$ and $|Q_{s}|<|Q_{s+1}|$ for each $s \geq 1$,
and let $k\geq 2$ and $P_1,P_2, \dots ,P_k$ be as in the assumption of Theorem \ref{th-Lower-R-general}.
Then, for every $\varepsilon > 0$ there exists $n_0$ such that,
whenever $n_k \ge n_0$, we can choose an integer $s$ with
\[
|Q_s|= O\left(\left(\frac{m_k}{n_k+td+2}\right)^{\frac{m}{n_{k-1}+\varepsilon}}\right),
\]
where we regard $k, t, n_1, \dots , n_{k-1}, m_1, \dots , m_{k-1}$ as constants.
In particular, for all such $n_k$ we have
$\operatorname{R}_{k,t}(\mathcal{Q}\,|\,P_1,P_2,\dots,P_k) > s$.

To prove Theorem \ref{th-Lower-R-general},
we use
the probabilistic result,
due to Lov\'{a}sz, which fundamentally improves the existence argument
in many instances.
Let $A_1,\ldots,A_n$ be events in a probability space $\Omega$. We
say that the graph $\Gamma$ with vertex set $\{A_1,A_2,\ldots,A_n\}$ is a
\emph{dependency graph} if 
two events $A_{i}$ and $A_j$ are adjacent if and only if $A_i$ and $A_{j}$ are dependent.
%
%
The following is a corollary of Lov\'{a}sz Local Lemma \cite{ErdosLovasz},
whose proof can be bound in \cite{GRS90}.

\begin{corollary}{\upshape \cite{GRS90}}\label{Lemma-LLL}
Suppose that $A_1,\ldots,A_n$ are events in a probability space
having dependency graph $\Gamma$.
If there exist positive numbers
$y_1,y_2,\ldots,y_n$ satisfying
$$
\ln y_i>\sum_{\{i,j\}\in \Gamma}y_j\Pr[A_j]+y_i\Pr[A_i],
$$
for $1\leq i\leq n$, then
$$
\Pr\left[\bigcap_{i=1}^n \overline{A_i}\right]>0.
$$
\end{corollary}

\subsection{Proof of Theorem \ref{th-Lower-R-general}}
\label{prooflower_sec}

\begin{proof}[Proof of Theorem \ref{th-Lower-R-general}]
Let
$a=\max\left\{{{n_i\choose t}\choose m_i}\,|\,1\leq i\leq k-1\right\}.$
In this proof, we denote by $e$ the Euler's constant.
Define the constant $C_0$ depending on $k,t,P_1,\dots ,P_{k-1}$ as

\begin{eqnarray*}
C_0 &=& \frac{(k-1)^{m-1}}{2ea} \left(\frac{n_{1}}{e}\right)^{n_{1}}.
\end{eqnarray*}
We assume
\begin{eqnarray*}
\ln n < \frac{m_k}{n_k+td+2},
\quad \text{and} \quad
(\ln n)^{m} n^{n_{k-1}} < C_0 \left(\frac{m_k}{n_k+td+2}\right)^m,
\end{eqnarray*}
and let $Q$ be a poset with $|Q| = n$.

If $n <n_{k}$,  then we can give a $k$-coloring of $t$-chains in $Q$ such that all $t$-chains are colored $k$. Since  $m = \min\{m_1,m_2,\dots,m_{k-1}\} \geq 1$, it follows that there is no monochromatic copy of $P_i$ in color $i$ for any $1\leq i\leq k$.
Thus,
we may assume $n \geq n_{k}$.

Let
$$p=\frac{n_k+td+2}{m_k} (\ln n),\label{equ2-7}$$
and let
$$p_1=\cdots=p_{k-1}=\frac{p}{k-1} = \frac{n_k+td+2}{(k-1)m_k} (\ln n)
\quad \text{and} \quad p_k=1-p.$$
Since
$\ln n < \frac{m_k}{n_k+td+2}$,
we see $0 < p_i < \frac{1}{k-1}$ for any $1 \leq i \leq k-1$ and $0< p_k<1$.
We color the $t$-chains in $Q$ so that
each $t$-chain is independently colored
by the color $i \ (1\leq i\leq k)$ with probability $p_i$.
For each $n_i$-element subset $S$ of $Q$, let $A_{S,i}$ be the event
that the subposet induced by $S$ contains a monochromatic copy of $P_i$ with color $i$.
Since the subposet induced by $S$ contains at most ${n_i \choose t}$ many $t$-chains and
at most ${{n_i \choose t}\choose m_i}$ copies of $P_i$,
we have
$$
\Pr[A_{S,i}]\leq {{n_i \choose t}\choose m_i}p_{i}^{m_i}=
\left\{
\begin{array}{ll}
{{n_i\choose t}\choose m_i}(\frac{p}{k-1})^{m_i} & \mbox{for~$1\leq i\leq k-1$,} \\[0.2cm]
{{n_k\choose t}\choose m_k}(1-p)^{m_k} & \mbox{for~$i=k$.}
\end{array}
\right.
$$
To prove Theorem \ref{th-Lower-R-general},
it suffices to show
$\Pr[\bigcap_{i=1}^k \bigcap \overline{A_{S,i}}]>0$,
where $\overline{A_{S,i}}$ is the complementary event of $A_{S,i}$
and
the second intersection takes over all $n_i$-element subsets $S$ of $Q$.

Let $\Gamma$ denote the dependency graph
with $\sum_{i=1}^k{n\choose n_i}$ vertices,
which correspond to all possible events $A_{S,i}$,
where there is an edge between $A_{S,i}$ and $A_{S',j}$
if and only if $S \cap S'$ contains a $t$-chain
(i.e., the events $A_{S,i}$ and $A_{S',j}$ are dependent).
For $1 \leq i,j \leq k$,
let $N_{A_iA_j}$ denote the maximum number of vertices of the form $A_{S',j}$
that are adjacent to $A_{S,i}$,
where the maximum takes over all $n_i$-element subsets $S$ of $Q$.

\begin{fact}
\label{fact_AiAj}
$$
N_{A_iA_j}+1 <
\left\{
\begin{array}{ll}
\left(\frac{e}{n_{1}}\right)^{n_{1}} n^{n_{k-1}}& \mbox{{\rm if}~$1\leq i\leq k$ {\rm and} $1\leq j\leq k-1$,}
\\[0.2cm]
\left(\frac{e}{n_k}\right)^{n_k} n^{n_k}
& \mbox{{\rm if}~$j=k$.}
\end{array}
\right.
$$
\end{fact}
\begin{proof}

Let $1\leq i\leq k$ and $1 \leq j \leq k-1$.
Recall that ${n\choose n_j} < \left(\frac{en}{n_{j}}\right)^{n_{j}}$. 
Since the number of $n_j$-element subsets of $Q$
is ${n\choose n_j}$ and $n\geq n_{k-1} \geq n_j \geq n_1$,
we have
$$
N_{A_iA_j}+1\leq {n\choose n_j}+1
< \left(\frac{en}{n_{j}}\right)^{n_{j}}
\leq \left(\frac{e}{n_{1}}\right)^{n_{1}} n^{n_{k-1}},
$$
where the last inequality is obtained by the fact that 
the function $f(x) = \left(\frac{e}{x}\right)^{x}$ is decreasing for $x \geq 1$.
For the case $j=k$,
we can similarly obtain the desired inequality,
which completes the proof of Fact \ref{fact_AiAj}.
\end{proof}

We set
$y_1=\cdots=y_{k-1}=e$ and $y_k=n$,
and will show that
the following holds
for each $i = 1,2, \dots,k$:
\begin{equation}
\ln y_i >
\sum_{j=1}^k \left\{y_j {{n_j\choose t}\choose m_j}p_j^{m_j} (N_{A_iA_j}+1)\right\}.
\label{equ2-1}
\end{equation}


Note that
for $i = 1,2, \dots,k$ and an $n_i$-element subset $S$ of $Q$,
we have
\begin{equation*}
\begin{split}
\sum_{j=1}^k \left\{y_j {{n_j\choose t}\choose m_j}p_j^{m_j} (N_{A_iA_j}+1)\right\}
&>\sum_{j=1}^k \left\{y_j {{n_j\choose t}\choose m_j}p_j^{m_j} N_{A_iA_j}\right\}
+y_i {{n_i\choose t}\choose m_i}p_i^{m_i}\\[0.2cm]
& \geq \sum_{\{A_{S,i},A_{S',j}\}\in E(\Gamma)}y_{j}\Pr[A_{S',j}]+y_i\Pr[A_{S,i}].
\end{split}
\end{equation*}
Therefore,
if Eq.~\eqref{equ2-1} holds for each $i = 1,2, \dots,k$,
then it follows from Corollary \ref{Lemma-LLL} that
$\Pr[\bigcap_{i=1}^k \bigcap \overline{A_{S,i}}]>0$,
that is,
there exists a $k$-coloring of the $t$-chains in $Q$
with no monochromatic copy of $P_i$ in color $i$ for any $1\leq i\leq k$,
as desired.
\medskip



Thus, it suffices to show Eq.~\eqref{equ2-1} for $1\leq i\leq k$.
Recall
$a=\max\left\{{{n_i\choose t}\choose m_i}\,|\,1\leq i\leq k-1\right\}$.
By Fact \ref{fact_AiAj},
\begin{eqnarray*}
\lefteqn{
\sum_{j=1}^k \left\{y_j {{n_j\choose t}\choose m_j}p_j^{m_j} (N_{A_iA_j}+1)\right\}
}\\[0.2cm]
&=& \sum_{j=1}^{k-1} \left\{y_j {{n_j\choose t}\choose m_j}p_j^{m_j} (N_{A_iA_j}+1)\right\}
+y_k {{n_k\choose t}\choose m_k}p_k^{m_k} (N_{A_iA_k}+1)\\[0.2cm]
&\leq& \sum_{j=1}^{k-1} \left\{
ea\left(\frac{p}{k-1}\right)^{m}
\left(\frac{e}{n_{1}}\right)^{n_{1}} n^{n_{k-1}}
\right\}
+ y_k {{n_k\choose t}\choose m_k}(1-p)^{m_k} \left(\frac{e}{n_k}\right)^{n_k} n^{n_k}\\[0.2cm]
&=&
\frac{ea p^m}{(k-1)^{m-1}}\left(\frac{e}{n_{1}}\right)^{n_{1}} n^{n_{k-1}}
+y_k {{n_k\choose t}\choose m_k}(1-p)^{m_k} \left(\frac{e}{n_k}\right)^{n_k} n^{n_k}.
\end{eqnarray*}

Since
$p=\frac{n_k+td+2}{m_k} (\ln n)$ and
$n_1\leq n_2\leq \dots \leq n_{k-1}$,
we can calculate the upper bound of the first term as
\begin{eqnarray*}
 \frac{ea p^m}{(k-1)^{m-1}} \left(\frac{e}{n_{1}}\right)^{n_{1}} n^{n_{k-1}}
&= &
 \frac{ea}{(k-1)^{m-1}}\left(\frac{e}{n_{1}}\right)^{n_{1}} n^{n_{k-1}}\left(\frac{n_k+td+2}{m_k} (\ln n)\right)^m\\[0.2cm]
&= &
\frac{ea}{(k-1)^{m-1}} \left(\frac{e}{n_{1}}\right)^{n_{1}}\left(\frac{n_k+td+2}{m_k}\right)^m
(\ln n)^m n^{n_{k-1}}
\\[0.2cm]
&\leq&
\frac{1}{2},
\end{eqnarray*}
where the last inequality is obtained by
$C_0 = \frac{(k-1)^{m-1}}{2ea} \left(\frac{n_{1}}{e}\right)^{n_{1}}$
and $(\ln n)^{m} n^{n_{k-1}} < C_0 \left(\frac{m_k}{n_k+td+2}\right)^m$.

We next consider the second term.
Since $d = {n_k\choose t}-m_k$ and $n \geq n_k$,
we have 
\begin{eqnarray*}
&&{{n_k\choose t}\choose m_k}={{n_k\choose t}\choose d}\leq {n_k\choose t}^d\leq n_k^{td}
\leq n^{td} = \exp\{td \ln n\},
\\\\[0.1cm]
&\text{and }
&
(1-p)^{m_k}\leq \exp\{-pm_k\} = \exp\{-(n_k+td+2) \ln n \}.
\end{eqnarray*}
Since $n_k \geq 3 > e$,
we have $\left(\frac{e}{n_k}\right)^{n_k} \leq 1$.
Thus,
$y_k = n= \exp\{\ln n\}$
and 
$n^{n_k} = \exp\left\{n_k \ln n \right\}$
imply
\begin{eqnarray*}
y_k {{n_k\choose t}\choose m_k}(1-p)^{m_k} \left(\frac{e}{n_k}\right)^{n_k} n^{n_k}
&<&
\exp\big\{\left(1+td-(n_k+td+2) + n_k\right)\ln n\big\}\\[0.2cm]
&=& \exp\{- \ln n\} \ = \ n^{-1} \ \leq \ \frac{1}{2}.
\end{eqnarray*}

Thus,
\begin{eqnarray*}
\lefteqn{
\sum_{j=1}^k \left\{y_j {{n_j\choose t}\choose m_j}p_j^{m_j} (N_{A_iA_j}+1)\right\}
}\\[0.2cm]
&\leq& 
\frac{ea p^m}{(k-1)^{m-1}}\left(\frac{e}{n_{1}}\right)^{n_{1}} n^{n_{k-1}}
+y_k {{n_k\choose t}\choose m_k}(1-p)^{m_k} \left(\frac{e}{n_k}\right)^{n_k} n^{n_k}
\\[0.2cm]
&\leq& \frac{1}{2} + \frac{1}{2} = 1 \leq \ln y_i,
\end{eqnarray*}
which shows Eq.~\eqref{equ2-1} for $1 \leq i \leq k$.
This completes the proof of Theorem \ref{th-Lower-R-general}.

\end{proof}


\section{Upper bounds for weak Boolean Ramsey numbers}
\label{M2_sec}

\subsection{Upper bounds obtained from Lubell functions}

For a family
$\mathcal{F}\subseteq B_N$, the \emph{Lubell function} of
$\mathcal{F}$ is defined as
$$
\operatorname{lu}_N(\mathcal{F})=\sum_{F\in \mathcal{F}}{N\choose
|F|}^{-1}.
$$
Note that $\operatorname{lu}_N(B_N) = N+1$.
As pointed out in \cite[P.~563]{CS18},
$\operatorname{lu}_N(\mathcal{F})$ can be interpreted
as the average size of $|\mathcal{F}\cap \mathcal{C}|$,
where $\mathcal{C}$ is a maximum chain in $B_N$,
i.e.~an $(N+1)$-chain in $B_N$.
Lubell functions of particular families have been much attracted
since they could be used to answer Tur\'{a}n-type problems in Boolean lattices,
see \cite{GriggsLi, JohnstonLuMilans} for instance.
%
By the definition,
we see that the Lubell function satisfies the linearity;
that is,
if $\mathcal{F}\cap \mathcal{G}= \emptyset$, then
\begin{equation}\label{eq-lubell-union}
\operatorname{lu}_N(\mathcal{F}\cup \mathcal{G})
=\operatorname{lu}_N(\mathcal{F}) +\operatorname{lu}_N(\mathcal{G}).
\end{equation}
In addition, we use the following theorem,
which sharpened the Sperner theorem
and is known as the YBLM-inequality \cite{Bollobas65, BK07, Lubell66, Meshalkin63, Yamamoto54}.
\begin{theorem}
\label{YBLM-ineq_thm}
Let $\mathcal{F}\subseteq B_N$ be an anti-chain.
Then, we have
\begin{equation*}\label{eq-lubell-bound}
\operatorname{lu}_N(\mathcal{F})=\sum_{F\in \mathcal{F}}{N\choose |F|}^{-1}
\leq 1.
\end{equation*}
\end{theorem}

For a poset $P$,
we say that a poset $Q$ is \emph{$P$-free} if there is no copy of $P$ in $Q$.
Let $L_N(P)$ be the maximum value of $\operatorname{lu}_N(\mathcal{F})$
among all $P$-free families $\mathcal{F}\subseteq B_N \setminus \{\emptyset, [N]\}$,
which was defined in \cite{CS18}.
That is,
\begin{center}
$L_N(P)=\max \{\operatorname{lu}_N(\mathcal{F})
:\mathcal{F}\subseteq B_N\setminus \{\emptyset, [N]\}$, $\mathcal{F}$
is $P$-free$\}$.
\end{center}
This definition can be generalized as follow. For a subset $\mathcal{Q}$ of $B_N$, let
\begin{center}
$L_N(P;\mathcal{Q})=\max \{\operatorname{lu}_N(\mathcal{F})
:\mathcal{F}\subseteq B_N\setminus \mathcal{Q}$, $\mathcal{F}$
is $P$-free$\}$.
\end{center}
If $\mathcal{Q}=\{\emptyset, [N]\}$, then $L_N(P;\mathcal{Q})=L_N(P)$.
Cox and Stolee \cite[Theorem 2.5]{CS18} gave an upper bound for the weak Boolean Ramsey number
in terms of $L_N(P)$.
We extend the result by using $L_N(P;\mathcal{Q})$ as follows,
and will give its application in the next subsection.


\begin{theorem}\label{th-condition}
Let $P$ be a poset, and let $\mathcal{Q}$ be a subset of $B_N$.
If
$$
k L_{N}(P;\mathcal{Q})<N+1-\operatorname{lu}_N(\mathcal{Q}),
$$
then
$\operatorname{R}_k(\mathcal{B}\,|\,P)\leq N$.
\end{theorem}
\begin{proof}
To show that $\operatorname{R}_k(\mathcal{B}\,|\,P)\leq N$,
suppose on the contrary that there is no monochromatic copy of $P$ in some $k$-coloring of $B_{N}$.
Let $\mathcal{F}_i$ be the family of $B_{N}\setminus \mathcal{Q}$
with color $i$, which is $P$-free. It follows from (\ref{eq-lubell-union}) that
$$
N+1-\operatorname{lu}_N(\mathcal{Q})>kL_{N}(P;\mathcal{Q})\geq \sum_{i=1}^k \operatorname{lu}_{N}(\mathcal{F}_i)=\operatorname{lu}_{N}\left(\bigcup_{i=1}^k\mathcal{F}_i\right)=
\operatorname{lu}_{N}\left(B_N\setminus \mathcal{Q}\right)=N+1-\operatorname{lu}_N(\mathcal{Q}),
$$
a contradiction.
\end{proof}

\subsection{Weak Boolean Ramsey numbers for matchings}
In this subsection
we evaluate $\operatorname{R}_k(\mathcal{B}\,|\,M_s)$ using Theorem \ref{th-condition}.
Recall that $M_s$ is the matching of size $s$.


\begin{theorem}\label{th-Ms}
Let $k, s$ be two integers with $k\geq 2$ and $s\geq 3$. Then,
$\operatorname{R}_k(\mathcal{B}\,|\,M_2)=k+2$
and
$k+2\leq \operatorname{R}_k(\mathcal{B}\,|\,M_s)\leq \max\{k+7,s\}$.
\end{theorem}

Recall that
for a poset $P$,
$e(P)$ is the largest integer $m$ such that $P$ cannot be embedded into any $m$ consecutive levels of $B_{N}$
for any non-negative integer $N$.
In other words,
$e(P)$ is the maximum $m$ such that
the family $\binom{[N]}{s}\cup \binom{[N]}{s+1}\cup \cdots
\cup \binom{[N]}{s+m-1}$ in $B_{N}$ does not contain a copy of $P$
for all non-negative integers $N$ and $s$.
The parameter $e(P)$ is commonly used for Tur\'{a}n-type problems in posets.
Cox and Stolee \cite[Question 6.2]{CS18} asked,
for a poset $P$, whether there is a constant $c=c(P)$ such that
$\operatorname{R}_k(\mathcal{B}\,|\,P) = k\cdot e(P)+c$.
Note that $e(M_s)=1$ for $s \geq 2$.
Thus, by Theorem \ref{th-Ms}, we have
$\operatorname{R}_k(\mathcal{B}\,|\,M_2)=k+2=k\cdot e(M_2)+2$
which implies that
the question by Cox and Stolee
is true for $P = M_2$ with $c = 2$.

We now evaluate $L_N(P)$ for the case $P = M_s$ with $s \geq 2$.
For a positive integer $s$,
the $(s,1)$-butterfly (resp.~the $(1,s)$-butterfly)
is particularly called the \emph{$s$-cup} (resp.~the \emph{$s$-cap})
and denoted by $\vee_s$ (resp.~$\wedge_s$).
The unique lowest element of $\vee_s$ (resp.~unique highest element of $\wedge_s$)
is called its \emph{center}.
Recall that $A_t$ is the $t$-anti-chain.

Let $P$ and $Q$ be posets on disjoint ground sets.
We write $P \sqcup Q$ for their disjoint union, that is, the poset whose
ground set is the union of the ground sets of $P$ and $Q$ and in which
no element of $P$ is comparable with any element of $Q$.
\begin{lemma}\label{lem2-2}
$L_N(M_2)=1+\frac{1}{N}$ for $N\geq 5$.
\end{lemma}
\begin{proof}
In $B_N\setminus \{\emptyset, [N]\}$,
we choose all elements in ${[N]\choose 2}$ and
an element in ${[N]\choose 1}$ to form $\mathcal{F}_0$.
Clearly, $L_N(M_2)\geq \operatorname{lu}_N(\mathcal{F}_0)=1+\frac{1}{N}$.

To show that $L_N(M_2)\leq 1+\frac{1}{N}$,
we prove $\operatorname{lu}_N(\mathcal{F}) \leq 1+\frac{1}{N}$
for any $M_2$-free family $\mathcal{F} \subseteq B_N \setminus \{\emptyset, [N]\}$.
Note that 
$\mathcal{F}$ forms either $\vee_s \sqcup A_t$,
$\wedge_s \sqcup A_t$,
or $C_3 \sqcup A_t$.
Suppose first that $\mathcal{F}$ forms $\vee_s \sqcup A_t$.
Suppose first $\mathcal{F}$ forms $\vee_s\cup A_t$.
Then let $\mathcal{X} = \mathcal{F} \setminus \{Y\}$
and $\mathcal{Y}=\{Y\}$,
where $Y \in B_N$ is the set corresponding to the center of $\vee_s$.
Note that $\operatorname{lu}_N(\mathcal{X}) \leq 1$ by Theorem \ref{YBLM-ineq_thm},
and $\operatorname{lu}_N(\mathcal{Y}) \leq \frac{1}{N}$ by the definition.
Thus, we have $\operatorname{lu}_N(\mathcal{F})
=\operatorname{lu}_N(\mathcal{X})+\operatorname{lu}_N(\mathcal{Y})\leq
1+\frac{1}{N}$ by (\ref{eq-lubell-union}).
The same is true for the case where $\mathcal{F}$ forms $\wedge_s\cup A_t$.

Thus, we may assume $\mathcal{F}$ forms $C_3\cup A_t$.
Let $X_1, X_2, X_3 \in B_N$ be the sets forming $C_3$ with $X_1 \subseteq X_2 \subseteq X_3$.
Let $\mathcal{A} = \mathcal{F} \setminus \{X_1, X_2, X_3\}$,
which forms the $t$-anti-chain $A_t$.
If $X_1\not\in {[N]\choose 1}$,
then $\operatorname{lu}_N(\mathcal{A} \cup \{X_3\})\leq 1$
 by Theorem \ref{YBLM-ineq_thm}
and $\operatorname{lu}_N(\{X_1, X_2\}) < \frac{2}{{N\choose 2}} \leq \frac{1}{N}$
since $N \geq 5$,
which implies $\operatorname{lu}_N(\mathcal{F}) =
\operatorname{lu}_N(\mathcal{A} \cup \{X_3\}) + \operatorname{lu}_N(\{X_1,X_2\})
\leq 1 +\frac{1}{N}$
by (\ref{eq-lubell-union}).
Thus, $X_1 \in {[N]\choose 1}$.
Similarly, we have $X_3\in {[N]\choose n-1}$.

Let $|X_2| = j$, where $j\in \{2,3,\dots ,N-2\}$,
and
let $\mathcal{Z}$ be the set of elements $Z \in {[n]\choose j}$
such that $X_1 \subseteq Z \subseteq X_3$.
Note that $X_2 \in \mathcal{Z}$ and $|\mathcal{Z}|={N-2\choose j-1}$.
Then
$$\operatorname{lu}_N(\mathcal{Z})={N-2\choose j-1}{N\choose j}^{-1}
= \frac{j(N-j)}{N(N-1)}\geq \frac{2(N-2)}{N(N-1)}\geq \frac{1}{N}+{N\choose
2}^{-1}\geq \frac{1}{N}+{N\choose j}^{-1}=\operatorname{lu}_N(\{X_1,X_2\}),$$
where the second inequality is obtained from $N\geq 5$.
Since any member $Z \in \mathcal{Z}$ satisfies $X_1 \subseteq Z \subseteq X_3$,
it is easy to see that $Z$ is comparable to no member of $\mathcal{A}$.
Thus, it follows from Theorem \ref{YBLM-ineq_thm} that
$\operatorname{lu}_N(\mathcal{A} \cup \mathcal{Z}) \leq 1$.
These imply
\begin{eqnarray*}
\operatorname{lu}_N(\mathcal{F})&=&
\operatorname{lu}_N(\{X_1,X_2\})+\operatorname{lu}_N(\mathcal{A})+\operatorname{lu}_N(\{X_3\})
\\[0.1cm]
&\leq &\operatorname{lu}_N(\mathcal{Z})+\operatorname{lu}_N(A_t)+\operatorname{lu}_N(\{X_3\})
=\operatorname{lu}_N(\mathcal{Z} \cup A_t)+\operatorname{lu}_N(\{X_3\})\leq 1+\frac{1}{N}.
\end{eqnarray*}
This completes the proof of Lemma \ref{lem2-2}.
\end{proof}

\begin{lemma}\label{lem3-6}
Let $s$ be an integer with $3\leq s\leq {N\choose 2}+1$,
and
let $\mathcal{Q}=\{\emptyset,[N]\}\cup {[N]\choose 1}\cup {[N]\choose N-1}$.
Then $1+\frac{2(s-1)}{N(N-1)}\leq L_N(M_s; \mathcal{Q})\leq 1+\frac{4(s-1)}{N(N-1)}$.
\end{lemma}
\begin{proof}
In $B_N\setminus \mathcal{Q}$,
we choose all sets in ${[N]\choose 3}$ and
$s-1$ sets in ${[N]\choose 2}$ to form $\mathcal{F}_0$.
Clearly, $L_N(M_s; \mathcal{Q})\geq \operatorname{lu}_N(\mathcal{F}_0)=1+\frac{2(s-1)}{N(N-1)}$.

To show that $L_N(M_s; \mathcal{Q})\leq 1+\frac{4(s-1)}{N(N-1)}$,
we prove $\operatorname{lu}_N(\mathcal{F}) \leq 1+\frac{4(s-1)}{N(N-1)}$
for any $M_s$-free family $\mathcal{F}\subseteq B_N\setminus \mathcal{Q}$.
Since $\mathcal{F}$ is $M_s$-free, 
there is $\mathcal{X} \subseteq \mathcal{F}$
such that $|\mathcal{X}| \leq 2(s-1)$ and $\mathcal{F}\setminus \mathcal{X}$ forms an anti-chain.
Therefore,
 $\operatorname{lu}_N(\mathcal{F} \setminus \mathcal{X})\leq 1$ by Theorem \ref{YBLM-ineq_thm},
and $\operatorname{lu}_N(\mathcal{X}) \leq \frac{4(s-1)}{N(N-1)}$
since each set in $\mathcal{X}$ has size at least 2 and at most $N-2$.
These imply $\operatorname{lu}_N(\mathcal{F})\leq 1+ \frac{4(s-1)}{N(N-1)}$.
\end{proof}


Now we are ready to prove Theorem \ref{th-Ms}.
\begin{proof}[Proof of Theorem \ref{th-Ms}]
We define a $k$-coloring $\chi$ of $B_{k+1}$ as follows:
For $S \in B_{k+1}$,
let $\chi(S) = i$ if $S\in {[k+1]\choose i}$ for $i \in [k]$,
let $\chi(S) = 1$ if $S = \emptyset \in {[k+1]\choose 0}$,
and let $\chi(S) = k$ if $S = [k+1] \in {[k+1]\choose k+1}$.
It is easy to check that no monochromatic copy of $M_s$ with $s \geq 2$ exists.
Thus, $\operatorname{R}_k(\mathcal{B}\,|\,M_s) \geq k+2$.

For the upper bound, we first prove the case of $M_2$.
For $k=2$,
it can be obtained by Theorem \ref{th-Walzer} (i).
Thus, we assume $k \geq 3$.
To show $\operatorname{R}_k(\mathcal{B}\,|\,M_2)\leq k+2$,
suppose contrary that there is no monochromatic copy of $M_2$ in
some $k$-coloring of $B_{k+2}$.
Since $k+2 \geq 5$,
it follows from Lemma \ref{lem2-2} that
$L_{k+2}(M_2) = 1 + \frac{1}{k+2}$.
Setting $\mathcal{Q} = \{\emptyset, [k+2]\}$,
where $\operatorname{lu}_{k+2}(\mathcal{Q}) = 2$,
we have $$
kL_{k+2}(M_2) = k+\frac{k}{k+2}<k+1 = (k+2) + 1 - \operatorname{lu}_N(\mathcal{Q}).
$$
By Theorem \ref{th-condition}, we have $\operatorname{R}_k(\mathcal{B}\,|\,M_2)\leq k+2$,
and hence $\operatorname{R}_k(\mathcal{B}\,|\,M_2)=k+2$.

Next, we give an upper bound for the case of $M_s$ with $s \geq 3$.
Let $N= \max\{k+7,s\}$.
To show $\operatorname{R}_k(\mathcal{B}\,|\,M_s)\leq N$,
suppose contrary that there is no monochromatic copy of $M_s$ in
some $k$-coloring of $B_{N}$.
Let $\mathcal{Q}=\{\emptyset,[N]\}\cup {[N]\choose 1}\cup {[N]\choose n-1}$.
Note that $\operatorname{lu}_N(\mathcal{Q})=4$.
By Lemma \ref{lem3-6},
$L_{N}(M_s;\mathcal{Q}) \leq 1 + \frac{4(s-1)}{N(N-1)}$.
By a simple calculation from $N = \max\{k+7,s\}$,
we obtain $N(N-1)(N-k-3) > 4k(s-1)$,
which implies that
$$
kL_{n}(M_s;\mathcal{Q}) \leq k+\frac{4k(s-1)}{N(N-1)}
<N-3 = N+1 - \operatorname{lu}_N(\mathcal{Q}).
$$
By Theorem \ref{th-condition}, we have $\operatorname{R}_k(\mathcal{B}\,|\,M_s)\leq N = \max\{k+7,s\}$.
\end{proof}


\section{Lower bounds for strong Boolean Ramsey numbers}
\label{lowerB_sec}

Conlon and Ferber \cite{CF21} and Wigderson \cite{Wi21} improved the longstanding lower bound for multicolor Ramsey
numbers of graphs. By replacing an explicit graph by a random graph, Sawin \cite{Sa22} made a further improvement.
In this section,
by adapting their methods,
we give a lower bound of the strong Boolean Ramsey numbers.
For positive integers $n$ and $t$ with $n \geq t$,
we denote by $h_n(t)$ the number of distinct $t$-chains in $B_{n}$.

\begin{theorem}\label{th-lower-Boolean}
Let $k,t, m_1, m_2,\ldots,m_k$ be integers with $k\geq 3$,
and 
$1 \leq t-1 \leq m_1\leq m_2 \leq \cdots \leq m_k$. Then
\begin{equation*}
\begin{split}
&\operatorname{R}_{k,t}^{\sharp}(\mathcal{B}\,|\,B_{m_1},B_{m_2},\ldots,B_{m_k})
>
\min\left\{
m_1 + \frac{h_{m_1}(t) +(h_{m_1}(t)-1)\log_2(k-1) -1}{4{m_{1}\choose\left\lfloor m_{1}/2\right\rfloor}}, \ 
m_k + \frac{h_{m_k}(t) -1}{2{m_{k}\choose\left\lfloor m_{k}/2\right\rfloor}} \right\}.
\end{split}
\end{equation*}
\end{theorem}

Note that we were able to determine the exact number for $h_n(t)$ as follows.

\begin{proposition}\label{th-La2}
Let $n,t$ be positive integers with $n \geq t$.
Then
$$
h_n(t)=\sum_{i=0}^{t-1}(-1)^{t-i+1}{t-1\choose i}(i+2)^n
= \Theta ((t+1)^n).
$$
\end{proposition}
\begin{proof}
Consider a $t$-tuple $(X_1,X_2,\ldots,X_t)$
such that $X_1\subseteq X_2\subseteq \cdots \subseteq X_t \subseteq [n]$.
For simplicity,
we let $X_0 = \emptyset$ and $X_{t+1} = [n]$.
Since each element in $[n]$ belongs to
exactly one of $X_j \setminus X_{j-1}$ for $1 \leq j \leq t+1$,
the number of such $t$-tuples is exactly $(t+1)^n$.
On the other hand,
such a $t$-tuple $(X_1,X_2,\ldots,X_t)$ forms a $t$-chain
if and only if
$X_j \setminus X_{j-1} \neq \emptyset$ for any $2 \leq j \leq t$.
Note that
\begin{itemize}
\item
For each $2 \leq j \leq t$,
exactly $t^n$ many $t$-tuples $(X_1,X_2,\ldots,X_t)$
satisfy that
$X_j \setminus X_{j-1} = \emptyset$,
since for such a $t$-tuple,
each element in $[n]$ belongs to
exactly one of $X_{j'} \setminus X_{j'-1}$ for $1 \leq j' \leq t+1$ with $j' \neq j$.
Note that there are ${t-1\choose 1}$ choices for such $j$.
\item
For each pair $2 \leq j_1 < j_2 \leq t$,
exactly $(t-1)^n$ many $t$-tuples $(X_1,X_2,\ldots,X_t)$
satisfy that
$X_{j_1} \setminus X_{{j_1}-1} = X_{j_2} \setminus X_{{j_2}-1} =\emptyset$
and there are ${t-1\choose 2}$ choices for such $j_1, j_2$.
\item
More generally,
for each $i$ integers $2 \leq j_1 < j_2 < \dots < j_i \leq t$,
exactly $(t-i+1)^n$ many $t$-tuples $(X_1,X_2,\ldots,X_t)$
satisfy that
$X_{j_1} \setminus X_{{j_1}-1} =
\cdots = X_{j_i} \setminus X_{{j_i}-1} =\emptyset$
and there are ${t-1\choose i}$ choices for such $j_1, \dots , j_i$.
\end{itemize}
This implies that
$$
h_n(t)
=(t+1)^n + \sum_{i=1}^{t-1} (-1)^{i}{t-1\choose i}(t-i+1)^n
=\sum_{i=0}^{t-1} (-1)^{i}{t-1\choose i}(t-i+1)^n.
$$
Then letting $i' = t-i-1$ in the summation,
together with
${t-1\choose i} = {t-1\choose t-1-i} = {t-1\choose i'}$,
we obtain the desired equality.
\end{proof}

By Proposition \ref{th-La2} and
${m\choose\left\lfloor m/2\right\rfloor} \sim \frac{2^{m+\frac{1}{2}}}{\sqrt{\pi m}}$,
Theorem \ref{th-lower-Boolean} gives 
$$
\operatorname{R}_{k,t}^{\sharp}(\mathcal{B}\,|\,B_{m_1},B_{m_2},\ldots,B_{m_k})
=
\Omega\left((\log_2 k) t^{m_1}+  t^{m_k}\right)
$$

To prove Theorem \ref{th-lower-Boolean}, we need several preliminary results.
For two integers $m$ and $N$ with $m \leq N$, let $a(m)$ be the number of distinct antichains in $B_m$, and $e(m, N)$ be the number of strong embeddings of $B_m$ into $B_N$.
Axenovich and Walzer \cite{AW17} gave the asymptotic behavior of $e(m,N)$
as
$$
\frac{N !}{(N-m) !}(a(m)-m)^{N-m} \leq e(m, N) \leq \frac{N !}{(N-m) !} a(m)^{N-m}.
$$
Since it is known that 
$a(m)=2^{(1+o(1)){n\choose \left\lfloor m/2\right\rfloor}}$
(see \cite{KM}) and $\frac{N !}{(N-m) !}\leq N^m=2^{m\log_2 N}$,
we obtain the following useful bound.
\begin{theorem}{\upshape\cite{AW17}}\label{number_b_n}
Let $m,N$ be two integers with $m \leq N$,
Then
$$
e(m, N) \leq 2^{(1+o(1)){m\choose \left\lfloor m/2\right\rfloor}(N-m) + m\log_2 N}
<
2^{2{m\choose \left\lfloor m/2\right\rfloor}(N-m)}.
$$
\end{theorem}


\medskip
In the rest of this section, we use the following notation.
For a positive integer $N$, let $\Tt$ be a set of $t$-chains in $B_N$.
We denote by $\Tc$ the complement of $\Tt$,
that is,
$\Tc$ is the set of all $t$-chains in $B_N$ that do not belong to $\Tt$.
For a strong embedding $f: B_m \rightarrow B_N$ of $B_m$ into $B_N$,
we say that \emph{$f(B_m)$ is embedded into $\Tt$}
if
$f(A)\in \Tt$ for any $t$-chain $A$ in $B_m$.
In addition, we say that $\Tt$ \emph{contains $B_m$}
if $\Tt$ contains $f(B_m)$ for some strong embedding $f$ of $B_m$ into $B_N$.
We denote by $e(m, \Tt)$ the number of strong embeddings of $B_m$ contained in $\Tt$.
Trivially $e(m,\Tt) \leq e(m,N)$
and $e(m,\mathcal{T}_0) = e(m,N)$, where $\mathcal{T}_0$ is the set of all $t$-chains in $B_N$.
Let
$$
c_t(m,\Tt)
= \dfrac{e(m,\Tt) }{e(m,N)}.
$$
Note that
$c_t(m,\Tt)$ means
for a strong embedding $f$ of $B_m$ into $B_N$ chosen uniformly at random,
the probability that $f(B_m)$ is embedded into $\Tt$.
Let
$$c_t(m,n,N)=\min \{c_t(m, \Tc)\,|\,\text{ $\Tt$ is a set of $t$-chains in $B_N$
that does not contain $B_{n}$}\}.
$$
We give an upper bound for $c_t(m,n,N)$.

\begin{lemma}\label{lem0}
For any $t \geq 3$ and $m, n\geq 2$, we have $c_t(m,n,N) \leq  \dfrac{e(m,N)\cdot 2^{-h_m(t) }}{1 - e(n,N) \cdot  2^{-h_n(t)}}$.
\end{lemma}
\begin{proof}
Let $\mathcal{T}$ be a set of $t$-chains in $B_N$
such that each $t$-chain is independently chosen with probability $\frac{1}{2}$.
Considering the expectation of the probability that $\overline{\mathcal{T}}$ contains $B_m$ under the condition that $\mathcal{T}$ does not contain $B_n$,
we have
\begin{equation*}\label{equ0}
\begin{split}
c_t(m,n,N)
&\leq \frac{\Pr[\text{$\Tc$ contains $B_m$}]}{\Pr[\text{$\mathcal{T}$ does not contain $B_{n}$}]}.
\end{split}
\end{equation*}
Since there are $e(m,\Tc)$ strong embeddings of $B_m$ contained in $\Tc$,
and $B_m$ has $h_m(t)$ many $t$-chains,
$
\Pr[\text{$\overline{\mathcal{T}}$ contains $B_m$}]
\leq e(m,N) \cdot 2^{-h_m(t)}.$
For the denominator,
we similarly have
$$\Pr[\mbox{$\mathcal{T}$ does not contain $B_{n}$}]
\geq 1 - e(n,N) \cdot  2^{-h_n(t)}.$$
Therefore,
$c_t(m,n,N) \leq \frac{e(m,N) \cdot 2^{-h_m(t) }}{1 - e(n,N) \cdot  2^{-h_n(t)}}$.
\end{proof}

We are now ready to prove Theorem \ref{th-lower-Boolean}.
\begin{proof}[Proof of Theorem \ref{th-lower-Boolean}]
Let
$$
N =\min\left\{
m_1 + \frac{h_{m_1}(t) +(h_{m_1}(t)-1)\log_2(k-1) -1}{4{m_{1}\choose\left\lfloor m_{1}/2\right\rfloor}}, \ 
m_k + \frac{h_{m_k}(t) -1}{2{m_{k}\choose\left\lfloor m_{k}/2\right\rfloor}} \right\}.
$$
For the lower bound, we construct a $k$-coloring of $B_N$ in the following way.
\begin{itemize}
\item By the definition of $c_t(m_1,m_k,N)$,
there is a set $\mathcal{T}_k$ such that
$\mathcal{T}_k$ contains no $B_{m_k}$ and
$\dfrac{e(m_1,\overline{\mathcal{T}_k}) }{e(m_1,N)} = 
c_t(m_1,\overline{\mathcal{T}_k}) = c_t(m_1,m_k,N)$.
Color all the $t$-chains in $\mathcal{T}_k$ by color $k$.
\item 
We color 
each $t$-chain in $\overline{\mathcal{T}_{k}}$
randomly with $1, 2,3,\ldots,k-1$, each with probability $1/(k-1)$, independently.
\end{itemize}
Note that there is no monochromatic induced copy of $B_{m_k}$ with color $k$. So we need to prove that the probability that there is a monochromatic induced copy of $B_{m_{i}}$ with some color $i \ (1\leq i\leq k-1)$
is less than $1$.

Let $i$ with $1\leq i \leq k-1$.
Recall that there are $e(m_i, \overline{\mathcal{T}_k})$ strong embeddings of $B_{m_i}$ into $\overline{\mathcal{T}_k}$.
Since each strong embedding of $B_{m_i}$ is monochromatic with color $i$ with probability $(k-1)^{-h_{m_{i}}(t)}$,
the probability that there exists a monochromatic $B_{m_i}$ is
at most
$e(m_i, \overline{\mathcal{T}_k}) \cdot (k-1)^{-h_{m_{i}}(t)}$.
Since $m_1 \leq m_2 \leq \dots \leq m_{k-1}$,
we have
$$e(m_i, \overline{\mathcal{T}_k}) \cdot (k-1)^{-h_{m_{i}}(t)}
\leq
e(m_1, \overline{\mathcal{T}_k}) \cdot (k-1)^{-h_{m_{1}}(t)}
=
c_{t}(m_{1},m_{k},N) \cdot e(m_1,N) \cdot  (k-1)^{-h_{m_{1}}(t)}.
$$

Therefore, it follows from Lemma \ref{lem0} that
\begin{equation*}
\begin{split}
&\Pr[\mbox{There exists a monochromatic induced copy of $B_{m_i}$ for some $1\leq i\leq k$}]
\\
&\leq \sum_{i=1}^{k-1} e(m_{i},\overline{\mathcal{T}_k}) \cdot (k-1)^{-h_{m_{i}}(t)}
\\[0.2cm]
&\leq \sum_{i=1}^{k-1}c_{t}(m_{1},m_{k},N) \cdot e(m_1, N) \cdot (k-1)^{-h_{m_{1}}(t)}\\[0.2cm]
&= c_{t}(m_{1},m_{k},N) \cdot e(m_1, N) \cdot (k-1)^{-h_{m_{1}}(t) +1}\\[0.2cm]
%
&\leq \frac{ 2^{-h_{m_1}(t) }}{1 - e(m_k,N) \cdot 2^{-h_{m_k}(t)}}
\cdot \{e(m_1, N)\}^2 \cdot (k-1)^{-h_{m_{1}}(t)+1}\\[2mm]
&< \frac{2^{-h_{m_1}(t) }}{1 - 2^{2{m_{k}\choose\left\lfloor m_{k}/2\right\rfloor}(N-m_k)} \cdot  2^{-h_{m_k}(t)}}\left(2^{4{m_{1}\choose\left\lfloor m_{1}/2\right\rfloor}(N-m_1)}\right)(k-1)^{-h_{m_{1}}(t)+1},
\end{split}
\end{equation*}
where the last inequality is obtained by Theorem \ref{number_b_n}.
By the choice for $N$, it is straightforward to calculate that
$\Pr[\mbox{There exists a monochromatic $B_{m_i}$ for some $1\leq i\leq k$}] < 1$.
Therefore, we have
$\operatorname{R}_{k,t}^{\sharp}(\mathcal{B}\,|\,B_{m_1},B_{m_2},\ldots,B_{m_k}) > N$.
\end{proof}

\section{Upper bounds for strong Boolean Ramsey numbers}

In this section,
we focus on $k$-coloring of sets in a Boolean lattice (i.e.~$t=1$)
and 
give upper bounds for the strong Boolean Ramsey numbers
for Boolean lattices $B_m$ and $r$-diamonds $\lozenge_r$
in Sections \ref{boo_subsec} and \ref{dia_subsec},
respectively.

\subsection{Boolean lattices}
\label{boo_subsec}

The following is the main theorem in this subsection.

\begin{theorem}\label{th-general-Boolean}
For $k\geq 6$ and $m\geq 2$, we have
$$
\operatorname{R}_k^{\sharp}(\mathcal{B}\,|\,B_{m})\leq (\operatorname{R}_{\lfloor k/2\rfloor}^{\sharp}(\mathcal{B}\,|\,B_{m})-2)\operatorname{R}_{\lceil k/2\rceil}^{\sharp}(\mathcal{B}\,|\,B_{m})+\operatorname{R}_{\lfloor k/2\rfloor}^{\sharp}(\mathcal{B}\,|\,B_{m}).
$$
\end{theorem}

We prove Theorem \ref{th-general-Boolean}
using the following lemma by Lu and Thompson \cite[Lemma 1]{LT22}.

\begin{lemma}[Lu and Thompson \cite{LT22}]
\label{le1}
Let $N, m_0, n_0, n', a, b$ be non-negative integers
satisfying
$n'\geq n_0 \geq a+b$ and
$$
N \geq (n_0+1-a-b)m_0 + n'.
$$ 
Then,
for any $2$-coloring with red and blue of the Boolean lattice $B_N$,
if there is a strong embedding $I: B_{n_0} \rightarrow B_N$
such that
$|I([n_0])| = n'$
and the following properties $(i)$ and $(ii)$ are satisfied:
\begin{itemize}
\item[] $(i)$
For all sets $S\in B_{n_0}$ with $|S| \leq a-1$, $I(S)$ is colored with blue;
\item[] $(ii)$
For all sets $S\in B_{n_0}$ with $|S| \geq n_0-b+1$,
$I(S) \cup\left([N] \setminus [n']\right)$ is colored with blue.
\end{itemize}
Then there exists 
a blue induced copy of $B_{n_0}$ or 
a red induced copy of $B_{m_0}$
in $B_N$.
\end{lemma}

\begin{proof}[Proof of Theorem \ref{th-general-Boolean}]
Let $n_0=\operatorname{R}_{\lfloor k/2\rfloor}^{\sharp}(\mathcal{B}\,|\,B_{m})$, and $m_0=\operatorname{R}_{\lceil k/2\rceil}^{\sharp}(\mathcal{B}\,|\,B_{m})$.
Note that $m_0\geq n_0 \geq k+1$.
Let  $N=(n_0-2)m_0+ n_0$. For any $k$-coloring of $B_N$, it suffices to prove that there is a monochromatic induced copy of $B_{m}$.

Note that ${N\choose 1} = 
N=(n_0-2)m_0+ n_0 \geq (n_0-2)(k+1)+ n_0>k(n_0-1)+1$.
Thus, 
there is a set $X \subseteq [N]$ with $|X| = n_0$ 
such that $\{x\}$ for any $x \in X$ is colored with color $i$ for some $1 \leq i \leq k$.
Let $C\subseteq [1,k]$ be a set of colors with $|C|=\lfloor k/2\rfloor \geq 3$
such that $C$ contains $i$, the color of $\emptyset$ and that of $[N]$.
%
\if
For some $q$ with $1 \leq q \leq k$,
there exist elements $A, A', X \in B_N$ with $A \subseteq A'$, $|A'| < k$,
$A'\cap X=\emptyset$, $|X|=m_q$
such that $A$ and $A' \cup \{x\}$ for any $x\in X$ are all colored with $q$.
\fi
Let $I : B_{n_0} \to B_N$ be an embedding
satisfying the following conditions;
\begin{itemize}
\item
$I(\emptyset)=\emptyset$.
\item
For $S \in B_{n_0}$ with $|S| = 1$,
we have $I(S) = \{x\}$ for some $x \in X$.
\item
For $S, S \in B_{n_0}$ with $|S| = |S'| = 1$ and $S \neq S'$,
we have $I(S) \neq I(S')$.
\item
For $S \in B_{n_0}$ with $|S| \geq 2$,
we have $I(S) = \bigcup_{s \in S} I(\{s\})$.
\end{itemize}

Let $n'= |I([n_0])| = |X| = n_0$.
We now regard the colors in $C$ as blue and 
all other colors as red. 
We now check that we can use Lemma \ref{le1} with $a=2$ and $b=1$.
By the choice of $X$,
$I(S)$ is colored with color in $C$ for all sets $S \in B_{n_0}$ with $|S| \leq 1$,
and
$I([n_0]) \cup ([N]\setminus I([n_0]) = [N]$ is colored with color in $C$.
Since
$N\geq (n_0-2) m_0 + n_0$,
it follows from Lemma \ref{le1} that 
there exists 
a blue induced copy of $B_{n_0}$ or 
a red induced copy of $B_{m_0}$ in $B_N$.

Suppose that there exists a blue induced copy of $B_{n_0}$.
Then all of the elements in the induced copy are colored with a color in $C$.
Since $|C|=\lfloor k/2\rfloor$ and $n_0=\operatorname{R}_{\lfloor k/2\rfloor}^{\sharp}(\mathcal{B}\,|\,B_{m})$,
there exists
a monochromatic induced copy of $B_m$ of color in $C$.
In the case where there exists a red induced copy of $B_{m_0}$,
since $m_0=\operatorname{R}_{\lceil k/2\rceil}^{\sharp}(\mathcal{B}\,|\,B_{m})$,
similarly we can show that there exists a monochromatic induced copy of $B_m$ of color not in $C$.
This completes the proof of Theorem \ref{th-general-Boolean}.
\end{proof}

\subsection{$r$-diamonds}
\label{dia_subsec}

By Theorem \ref{th-Walzer} (ii),
we have $2k
\leq \operatorname{R}_k^{\sharp}(\mathcal{B}\,|\,B_2)=\operatorname{R}_k^{\sharp}(\mathcal{B}\,|\,\lozenge_2)\leq 3^k-1$.
In this section,
we prove the following improvement,
where we obtain
$\operatorname{R}_k^{\sharp}(\mathcal{B}\,|\,\lozenge_2) \leq 5k-3$
by substituting $r=2$.

\begin{theorem}\label{BR22-B2}
Let $k,r$ be positive integers with $r \geq 2$. Then
$$
2k \leq \operatorname{R}_k^{\sharp}(\mathcal{B}\,|\,\lozenge_r)\leq 3kr-2r -k +1.
$$
\end{theorem}
\begin{proof}
The lower bound follows from Theorem \ref{th-Walzer} (ii).
In fact, we can construct a $k$-coloring on $B_{2k-1}$:
all elements in ${[2k-1] \choose 2i-2} \cup {[2k-1] \choose 2i-1}$ are colored
with $i$ for $i = 1,2,\dots, k$.
One can see that there is no monochromatic induced copy of $\lozenge_r$.

We now prove the upper bound.
Let $N = 3kr-2r -k +1$,
and consider a $k$-coloring of the Boolean lattice $B_{N}$.
It suffices to show that
this $B_{N}$ contains a monochromatic induced copy of $\lozenge_r$.
We will inductively define the sets $X_i$'s and $Y_i^j$'s
satisfying the following properties:
\begin{itemize}
\item[(1)]
$X_i \in {[N] \choose ir}$ for any $0 \leq i \leq 2k-1$,
\item[(2)]
$Y_i^j \in {[N] \choose (i-1)r+1}$
such that
$|Y_i^j \setminus X_{i-1}| =1$
for any $1 \leq i \leq 2k-1$ and any $1 \leq j \leq r$,
\item[(3)]
for each $1 \leq i \leq 2k-1$,
$Y_i^1, Y_i^2, \dots , Y_i^r$ have the same color,
\item[(4)]
$X_0 \subseteq Y_{i_1}^1 \subseteq Y_{i_2}^j \subseteq Y_{i_3}^1 \subseteq X_{2k-1}$
for any $1 \leq i_1 < i_2 < i_3 \leq 2k-1$
and any $1 \leq j \leq r$.
\end{itemize}
To find such sets, first let $X_0 = \emptyset$.
For $1 \leq i \leq 2k-1$,
suppose that $X_{i-1} \in {[N] \choose (i-1)r}$ is already defined.
Since $i \leq 2k-1$, we have $|X_{i-1}| = (i-1)r \leq (2k-2)r$,
and hence the number of elements in ${[N] \choose (i-1)r+1}$
that include $X_{i-1}$ is at least $N - (2k-2)r
= k(r-1)+1$.
Thus, $r$ of such elements have the same color,
and we define such elements as $Y_i^j$ with $1 \leq j \leq r$.
Note that $|Y_i^j \setminus X_{i-1}| =1$ for $1 \leq j \leq r$.
In addition,
let $X_i = \bigcup_{j=1}^r Y_i^j$,
which is an element of ${[N] \choose ir}$.
It is easy to see that
by performing the above procedure iteratively,
we obtain sets satisfying the conditions (1)--(4).

For $1 \leq i \leq 2k-1$,
it follows from the condition (3) that
all of $Y_i^1, Y_i^2, \dots , Y_i^r$ have the same color, say $c_i$.
Let $c_0$ be the color of $X_0$,
and $c_{2k}$ be the color of $X_{2k-1}$.
Then,
there exist three indices $i_1 < i_2 < i_3$ in $\{0, 1, \dots , 2k\}$
such that $c_{i_1} = c_{i_2} = c_{i_3}$.
If $0 \neq i_1$ and $i_3 \neq 2k$,
then it follows from the condition (4) that
$Y_{i_1}^1, Y_{i_2}^j$ with $1 \leq j \leq r$ and $Y_{i_3}^1$
form a monochromatic induced copy of $\lozenge_r$
with color $c_{i_1}$.
In the case when $i_1 = 0$ and/or $i_3 = 2k$,
using $X_{0}$ and/or $X_{2k-1}$
instead of $Y_{i_1}^1$ and/or $Y_{i_3}^1$,
we similarly obtain a monochromatic induced copy of $\lozenge_r$
with color $c_{i_1}$.
This completes the proof of Theorem \ref{BR22-B2}.
\end{proof}

\section*{Acknowledgment}

The second author visited the University of Hamburg in 2023 and would like to express his gratitude to Professor Mathias Schacht for the discussions he provided. The authors also thank Gang Yang for his helpful discussions related to this paper.




\end{document}